\documentclass[12pt]{amsart}%
\usepackage{amsmath}
\usepackage{amssymb}
\usepackage{amsthm}
\usepackage{amscd}
\usepackage[dvipdfmx]{graphicx}
\usepackage[mathscr]{eucal}
\usepackage{bm}
\usepackage{xcolor}
\makeindex
\newtheorem{Thm}{Theorem}[section]
\theoremstyle{definition}
\newtheorem{Theorem}[Thm]{Theorem}
\newtheorem{Lemma}[Thm]{Lemma}
\newtheorem{Corollary}[Thm]{Corollary}
\newtheorem{Proposition}[Thm]{Proposition}

\newtheorem{Example}[Thm]{Example}
\theoremstyle{remark}
\newtheorem{Remark}{Remark}

\font\ym=msbm10  
\newcommand{\Aut}{{\rm Aut}}

\newcommand{\R}{\text{\ym R}}

\newcommand{\T}{\text{\ym T}}
\newcommand{\C}{\text{\ym C}}

\newcommand{\sN}{\mathscr N}

\newcommand{\End}{\hbox{\rm End}}

\title[endofunction]{M\"obius transformation and the integral representation of endofunctions}
\begin{document}
% Haagerup positive definite functions on free groups.
% To Nobuko with our deep gratitudes. 
\maketitle
\begin{center}
  Tomohiro Hayashi (Nagoya Institute of Technology)
\end{center}
\begin{center}
  and
\end{center}
\begin{center}
  Shigeru Yamagami (Nagoya University)
\end{center}
% \begin{center}
% Graduate School of Mathematics
% \end{center}
% \begin{center}
% Nagoya University 
% \end{center}
% \begin{center} 
% Nagoya, 464-8602, JAPAN
% \end{center} 

\subjclass{MSC2020: 46G99, 31A10}
% 46G FA Measures, integration, derivative, holomorphy
% 31A10 PT Integral representations, integral operators, integral equations methods in two dimensions 

\keywords{Pick function, M\"obius transformation, integral representation, operator monotone function}

% Keywords: Haagerup function, Stieltjes inversion formula, polynomial hypergroup, moment problem,
% analytic linear functional, 
% 46L99 (operator algebra), 46J25(commutative Banach algebra), 43A99(abstract harmonic analysis) 
% Related to radial functions on free groups, we focus on certain polynomial hypergroups and
% work out spectral analysis with the help of the Stieltjes transform of their analytic functionals.
\begin{abstract}
  We supplement the Herglotz-Nevanlinna integral representation of so-called Pick functions by adding the formula for
  M\"obius transforms and the positivity characterization near boundary supports. 
\end{abstract}

\section*{Introduction}
Given a simply connected region $G$ in $\C$,
we say that a holomorphic function $\varphi$ on $G$ is an endofunction if $\varphi(G) \subset G$.
Functions of this type have been long and deeply investigated by many researchers from various points
of view, which are referred to under the names such as
Carath\'eodory, Herglotz, Nevanlinna, Pick and Schur in literature.

When $G$ is the upper half plane $\C_+ = \{ z \in \C; \text{Im}\,z > 0\}$,
a remarkable fact is that, if an atomic endofunction $\phi_s$ of $\C_+$ is assigned to each
boundary point $s \in \overline{\R} = \R \sqcup \{\infty\}$
inside the Riemann sphere $\overline{\C} = \C \sqcup \{\infty\}$ by
\[
  \phi_s(z) = \frac{1+sz}{s-z}, 
\]
then any endofunction is expressed up to additive real constants
by the integral of $\phi_s$ with respect to a unique positive measure on the boundary $\overline{\R}$.
In view of the fact that $\phi_s(z)$ ($z \in \C \setminus \R$) is total
in the continuous function space $C(\overline{\R})$
as a family of functions of $s$, the representing measure $\lambda$ is uniquely determined by $\varphi$.
More explicitly, we have the Stieltjes inversion formula 
\[
  \lambda(dx) = \lim_{y \downarrow 0} \frac{\text{Im}\,\varphi(x+iy)}{\pi(1+x^2)}\,dx
\]
as a Radon measure on $\R$ and
\[
  \lambda(\{\infty\}) = \lim_{y \to \infty} \frac{\text{Im}\,\varphi(iy)}{y}. 
\]
From this formula, one sees that $\lambda(U) = 0$ for an open neighborhood $U$ of 
a boundary point $c \in \overline{\R}$ if and only if
$\varphi$ is continuously extended to $U$ in such a way that
$\lim_{y \downarrow 0} \varphi(x+iy) \in \R$ ($x \in U$).
Thus, if we define the boundary support $[\varphi]$ of a holomorphic function $\varphi$ on $\C_+$
to be the residual of such points $c$, 
% by the complement of the set of $c$'s,
$[\lambda] = [\varphi]$. 
(There are so many literature on the subject and we just nominate \cite{Do, RR, Si, Vl} as references.)

In this paper, as a continuation of \cite[Appendix]{Y}, we shall supplement properties of endofunctions
in connection with M\"obius transformation and boundary positivity:

Let $\varphi$ be an endofunction with the representing measure $\lambda$.
Consider an invertible $2\times 2$ matrix $M$ for which 
the restriction $\phi_M$ of the associated linear fractional function to $\C_+$ is an endofunction and call $M$ an endomatrix. 
The representing measure $\mu_M$ of the composite endofunction $\varphi\circ\phi_M$ is then
explicitly described in terms of $\lambda$ and $M$.
Furthermore, the measure $\mu_M$ is parametrized by boundary points of $\overline{\R}$ to have 
to a Markov operator $\Lambda_M$ on $C(\overline{\overline{\R}})$ in such a way that $\Lambda_{MN} = \Lambda_M \Lambda_N$
for endomatrices $M$ and $N$. 

As for boundary positivity, a holomorphic function different from a real constant
is well-known to be
an endofunction if $\text{Im}\,\varphi(z) \geq 0$ for $z \in \C_+$ near the boundary $\overline{\R}$.
We point out here that this boundary positivity follows
from positivity near the boundary support $[\varphi]$, which makes the positivity checking easier.

The authors are grateful to M.~Nagisa and M.~Uchiyama for fruitful discussions
and helpful comments on the present work.

\section{Background}
Let $D = \{ z \in \C; |z| < 1\}$ be the unit disk in the complex plane with the boundary denoted by $\T = \{ z \in \C; |z| = 1\}$.
Thanks to Herglotz and Riesz, a holomorphic function $\phi$ on $D$ satisfying $\text{Re}\, \phi(z) \geq 0$ ($z \in D$) is characterized by
the integral representation
\[ 
\phi(z) = i\text{Im}(\phi(0)) 
+ \frac{1}{2\pi}\int_\T \frac{\zeta + z}{\zeta - z}\, \mu(d\zeta), 
\] 
where $\mu$ is a  positive Radon measure on $\T$ and uniquely determined by $\phi$.
% Note that, unless $\mu=0 \iff \phi(z) \equiv i \text{Im}\, \phi(0)$, $\text{Re}\, \phi(z) > 0$ ($z \in D$).
% Let us call $\mu$ the boundary measure of $\phi$. 

By the holomorphic change-of-variable
\[ 
  % D \ni z \mapsto
  w = i\frac{1-z}{1+z}, 
\quad 
z = \frac{i-w}{i+w}, 
\]
the unit disk $D = \{ |z| < 1\}$ is biholomorphically mapped to
the upper half plane $\C_+$ and vice versa.
Letting $\varphi(w) = i\phi(z)$, we obtain a holomorphic function $\varphi$ on $\C_+$
with its range controlled by $\text{Im}\,\varphi(w) \geq 0$ ($w \in \C_+$), which implies % so that
$\varphi(w) \in \C_+$ ($w \in \C_+$) unless $\varphi$ is a real constant.

In terms of the obvious parametrization $s \in \R \sqcup\{\infty\}$ of the boundary $\overline{\R}$ of $\C_+$
in the Riemann sphere $\overline{\C} = \C \sqcup \{\infty\}$, the integral representation of $\phi$ is rephrased by
\[
  \varphi(w) = \text{Re}\,\varphi(i) + \int_{\overline{\R}} \frac{1+sw}{s-w}\, \lambda(ds), 
\]
where $\lambda$ is the measure transferred from $\mu$ under the homeomorphism 
\[
  s \longleftrightarrow e^{it} \quad
  \text{with $s$ and $t$ related by $s = \tan \frac{t}{2}$ ($-\pi < t \leq \pi)$}
\]
and called the \textbf{representing measure} of $\varphi$.
By separating the mass at the point $\infty$ from $\lambda$, the above representation takes the form 
\[
  \varphi(w) = \text{Re}\,\varphi(i) + \lambda(\{\infty\}) w + \int_\R \frac{1+sw}{s-w}\, \lambda(ds). 
\]
Notice here that $\text{Im}\,\varphi(i) = \lambda(\overline{\R}) = \lambda(\{\infty\}) + \lambda(\R)$.

% The boundary measure $\lambda$ is determined by $\varphi$ as follows: For the point mass $\lambda(\{\infty\})$ at $\infty$, 
% it is read off from $\varphi$ by 
% \[
%  % \beta % = \lim_{v \to \pm\infty} \frac{\psi(u+iv)}{iv}
%   \lambda(\{\infty\})
%   = \lim_{v \to \pm\infty} \frac{\text{Im}\, \varphi(u+iv)}{v}
%   \quad
%   (\forall u \in \R), 
% \]
% whereas the restriction of $\lambda$ to $\R$ is recovered by 
% \[
%  % \lim_{v \to +0} \text{Im}\,\varphi(u+iv)du = \pi (u^2+1) \lambda(du) \iff
%   \lambda(du) = \lim_{v \to +0} \frac{\text{Im}\,\varphi(u+iv)}{\pi(u^2+1)} du 
% \]
% % or equivalently
% % \[
% % \pi \lambda(du) = \pi \beta \delta_\infty +  \lim_{v \to +0} \frac{\text{Im}\,\psi(u+iv)}{u^2+1} du
% % \]
% (Stieltjes inversion formula) as a weak* limit of positive linear functionals on $C_c(\R)$.  

\begin{Example}
  Let $0 \not= a \in \C$ and $b,c \in \C$. 
  \begin{enumerate}
  \item
    A linear fractional function 
  \[
  \varphi(z) = \frac{a}{z+c} + b
    \quad
   (z \in \C_+)
  \]
  is an endofunction % takes values in $\R \sqcup \C_+$
  if and only if $\text{Im}\,b \geq 0$, $\text{Im}\,c \geq 0$ and
  \[
    |a| + \text{Re}\,a \leq 2 (\text{Im}\, b) (\text{Im}\, c).
  \]
  %In this case, the representing measure of $\phi$ depends continuously on $(a,b,c)$ with respect to the weak* topology. Why?
\item
  An affine function $\varphi(z) = az+b$ ($z \in \C_+$) is an endofunction if and only if
  $a>0$ and $\text{Im}\,b \geq 0$.
\end{enumerate}
\end{Example}

\begin{proof}
  (i) As the pole $z = -c$ of $\varphi$ is outside of $\C_+$, we should have $\gamma = \text{Im}\,c \geq 0$.
  The behavior of $\varphi(z)$ near $z=\infty$ then necessitates $\beta = \text{Im}\,b \geq 0$.

  When $\gamma = 0$ (i.e., $c \in \R$), the behaviou of $\varphi(z)$ near $z = -c$ necessitates $a < 0$ % for $\phi$ being an endofunction,
  which in turn implies $\varphi \in \End(\C_+)$.
  
  For $\gamma > 0$, 
  $\{a/(z+c); z \in \C_+\}$ is an open disk of radius $|a|/2\gamma$ centered at $-ia/2\gamma$,
  whence $\varphi$ is an endofunction 
  % the value $\phi(z)$ ($z \in \C_+$) remains within $\R \sqcup \C_+$
  if and only if $|a| + \text{Re}\,a \leq 2\beta\gamma$. The last condition is equivalent to $a <0$ even for $\gamma = 0$. 
  
  Remark that, if $\beta\gamma > 0$, 
  \[
    |a| + \text{Re}\,a \leq 2\beta\gamma \iff |a|^2 \leq (2\beta\gamma - \text{Re}\,a)^2, 
  \]
  which is further rephrased by the positive semidefiniteness of 
  \[
    \text{Im}\,\varphi(x- \text{Re}\,c) = \frac{\beta x^2 + x\text{Im}\,a - \gamma \text{Re}\,a + \beta\gamma^2}{x^2 + \gamma^2}
    \quad
    (x \in \R). 
  \]
  
  When $\not= a \in \R$ and $b,c \in \R$,
  the positivity $\text{Im}\,\varphi(z) > 0$ ($z \in \C_+$) is equivalent to $|a| + \text{Re}\, a \leq 0$, i.e., $a < 0$,
  and the expression 
  \[
    \varphi(z) = b + \frac{ac}{1+c^2} + \frac{-a}{1+c^2} \frac{1-cz}{-c - z}
  \]
 reveals that the boundary measure of $\varphi$ is given by an atomic measure $(-a)/(1+c^2) \delta(s + c)\, ds$ at $-c$.  

 (ii) An affine function $az+b$ satisfies $\text{Im}\,(az+b) \geq 0$ ($z \in \C_+$) if and only if $a\geq 0$ and $\text{Im}\,b \geq 0$.
  The boundary measure $\lambda$ of $\varphi$ is given by $\lambda(\{\infty\}) = a$ and
  $\displaystyle \lambda(dx) = \frac{\text{Im}\,b}{\pi(1+x^2)}\, dx$ on $\R$
  with the real constant term given by $\text{Re}\,b$.
\end{proof}

 A holomorphic function $f$ on
 $\C \setminus \R = \overline{\C} \setminus \overline{\R}$ is said to be \textbf{real}
 if $\overline{f(\overline{z})} = f(z)$ ($z \in \C \setminus \R$).
Any holomorphic function $f_+$ on $\C_+$ is extended to real one $f$ in a one-to-one fashion and
the \textbf{boundary support} of $f_+$ or $f$, denoted by $[f_+]$ or $[f]$, is defined to be the complement of
\[
\{ t \in \overline{\R}; \text{$f$ is holomorphically extended to a neighborhood of $t$ in $\overline{\C}$} \}
\]
in $\overline{\R}$. Notice that the boundary support is a closed subset of $\overline{\R}$. 

By Stieltjes inversion formula and Schwarz reflection principle,
the support of the representing measure $\lambda$ of $\varphi$ is exactly the boundary support of $\varphi$.

In view of the correspondence $\varphi(w) = i\phi(z)$ ($z \in D$) discussed at the beginning,
we introduce the boundary support $[f]$ of a holomorphic function
$f$ on $D$ through a purely imaginary extension specified by $\overline{f(\overline{z})} = - f(1/z)$ ($z \in \overline{\C} \setminus \T$):
The complement of $[f]$ in $\T$ is the set of $\zeta \in \T$ for which $f$ is holomorphically extended to a neighborhood of $\zeta$ in $\C$.

Recall that, thanks to the maximum principle for harmonic functions,
a holomorphic function $\varphi$ on $\C_+$ satisfies $\text{Im}\,\varphi(z) \geq 0$ ($z \in \C_+$) if and only if the inequality holds
on a neighborhood of $\overline{\R}$ in $\C_+$ (boundary positivity).

From the definition of boundary support, a holomorphic function $\varphi$ on $\C_+$ is holomorphically extended to
$\overline{\C} \setminus [\varphi]$ and takes real values on $\overline{\R} \setminus [\varphi]$ by the reality of extension.
Thus, if $[\varphi] \not= \overline{\R}$, $\varphi$ is real analytic on any open interval $I$ disjoint from $[\varphi]$.

For $\varphi$ fulfilling the condition $\text{Im}\,\varphi(z) \geq 0$ ($z \in \C_+$), 
the integral representation reveals that the restriction $\varphi_I$ to $I$ is an operator monotone function.
Conversely any operator monotone function $f$ on an open interval $I$ is analytically extended to
a holomorphic function $\varphi$ on $I \sqcup (\C \setminus \R)$ so that $\text{Im}\,\varphi(z) \geq 0$ ($z \in \C_+$) and
$I \cap [\varphi] = \emptyset$ by L\"owner's characterization.

\section{M\"obius Endofunctions}
% Holomorphic functions of controlled ranges are referred to under
% various names such as Carath\'eodory, Schur, Herglotz, Pick and Nevanlinna.
% In what follows, we use rather neutral terminology:
Recall that an \textbf{endofunction} is a holomorphic function of $\C_+$ into itself.
%is called an \textbf{endofunction} and
We denote the set of endofunctions on $\C_+$ by $\End(\C_+)$,  
which forms a unital semigroup by composition and is closed under addition and
positive scalar multiplication.
Notice that $\C_+$ is included in $\End(\C_+)$ as constant functions
so that $\C_+ + \End(\C_+) \subset \End(\C_+)$.

An endofunction $\varphi$ is called an \textbf{autofunction}
if $\varphi$ maps $\C_+$ onto $\C_+$ bijectively. 
Let $\Aut(\C_+)$ be the set of autofunctions, which forms a group by composition. 

For an invertible matrix $A =
\begin{pmatrix}
  a & b\\
  c & d
\end{pmatrix}$ of complex entries, i.e., 
for $A \in \text{GL}(2,\C)$, the associated M\"obius transform
(also known as a linear fractional transform) of $z \in \overline{\C}$ is denoted by 
\[
  A.z = \frac{az+b}{cz + d}. 
\]
An invertible matrix $A$ is called an \textbf{endomatrix} if $A.z \in \C_+$ ($z \in \C_+$) and 
let $\text{GL}_+(2,\C)$ be the set of endomatrices, % $A \in \text{GL}(2,\C)$ satisfying $A.z \in \C_+$ ($z \in \C_+$),
which is a subsemigroup of the matrix group $\text{GL}(2,\C)$.
When $A$ is a matrix of real entries,
the condition $A \in \text{GL}_+(2,\C)$ is simplified to $\det(A) > 0$ and
$\text{GL}_+(2,\R) \equiv \text{GL}_+(2,\C) \cap \text{GL}(2,\R)$ is a subgroup
(the connected component) of $\text{GL}(2,\R)$. 
This can be seen in the following strengthened form: 
% Let $\Aut(\C_+)$ be the group of endofunctions which map $\C_+$ bijectively onto $\C_+$.
Thanks to the Schwarz lemma, each $\varphi \in \Aut(\C_+)$ is of the form $\varphi(z) = A.z$ with $A \in \text{GL}_+(2,\R)$. 
% realized by a M\"obius transform
% \[
%   \varphi(z) = A.z = \frac{az+b}{cz+d},
%   \quad
%  A = \begin{pmatrix}
%    a & b\\
%    c & d
%  \end{pmatrix}
% \]
% with a $2\times 2$ matrix $A$ of real entries satisfying $\det(A) > 0$.

Since proportional matrices give rise to the same transform,
we may restrict these invertible matrices to the subsemigroup $\text{SL}_+(2,\C) \equiv \text{GL}_+(2,\C)\cap\text{SL}(2,\C)$
to obtain linear fractional endofunctions.
In particular, $\Aut(\C_+)$ is the image of the double covering homomorphism $\text{SL}(2,\R) \to \Aut(\C_+)$ by M\"obius transformation,
which induces an obvious biaction of $\text{SL}(2,\R)$ on $\End(\C_+)$:
\[
  (A\varphi)(z) = A.\varphi(z) = \frac{a\varphi(z)+b}{c\varphi(z)+d},
  \quad
  (\varphi A)(z) = \varphi(A.z) = \varphi\left(\frac{az+b}{cz+d}\right).
\]

M\"obius transforms of upper triangular matrices in $\text{SL}(2,\R)$ consist of affine endofunctions $az+b$ ($a>0$, $b \in \R$)
and form a subgroup of $\Aut(\C_+)$, which is naturally isomorphic to the semidirect product $\R\rtimes \R_+$.
As an affine transformation group, $\R\rtimes\R_+$ acts on $\C_+$ freely and transitively. 
The left composition by affine endofunctions therefore gives rise to the left action of the affine transformation group
$\R\rtimes\R_+$ on $\End(\C_+)$ for which we can choose endofunctions satisfying $\varphi(i) = i$ as representatives of orbits.
Clearly the set $K$ of these representatives is a convex subset of $\End(\C_+)$ with its extremal points given by 
\[
  \phi_s(z) = 
  \begin{pmatrix}
    s & 1\\
    -1 & s
  \end{pmatrix}.z = \frac{1+sz}{s-z}
  \quad
  (s \in \overline{\R}), 
\]
which will be called an \textbf{atomic endofunction}. Here $\phi_\infty(z) = z$ by definition. 
Notice that an autofunction $\varphi \in \Aut(\C_+)$ is atomic if and only if $\varphi(i) = i$. 

% In view of
% \[
%   (\varphi A)(i) = \varphi\left(\frac{az+b}{cz+d}\right).
% \]

We now go into how the integral representation of an endofunction is changed under the right action of
$\text{SL}(2,\R)$.

Consider an autofunction $A.z \in \Aut(\C_+)$ with $A = \begin{pmatrix} a & b\\ c & d\end{pmatrix} \in \text{GL}_+(2,\R)$,
which is also denoted by $\phi_A$. In view of 
\[
  A.i = \frac{ac + bd + i(ad-bc)}{c^2+d^2}, 
\]
one sees that
\[
  A.z = \frac{ac+bd}{c^2+d^2} + \frac{ad-bc}{c^2+d^2} \phi_s(z),
  \quad
  s = - \frac{d}{c} 
\]
for $c \not= 0$ and $A.z = (az+b)/d$ for $c=0$. 
By replacing $A$ with $\begin{pmatrix} s & 1\\ -1 & s\end{pmatrix} A$, we have the following expression for the autofunction
$\phi_sA$: 
\begin{align*}
  (\phi_sA)(z) &=  \left( \begin{pmatrix} s & 1\\ -1 & s\end{pmatrix} A\right).z
  = \frac{(as+c)z + bs + d}{(-a+cs)z -b + ds}\\
             &=\text{Re}\, (\phi_sA)(i) % \frac{(ac+bd)(s^2-1) + (c^2+d^2 - a^2 - b^2)s}{(-a+cs)^2 + (-b+ds)^2} 
               + \frac{(ad-bc)(1+s^2)}{(-a+cs)^2 + (-b+ds)^2}\, \phi_t(z), 
\end{align*}
where
\[
  t = \frac{b-ds}{cs-a} =
  \begin{pmatrix}
    d & -b\\
    -c & a
  \end{pmatrix}. s
  \iff
  s = A.t
\]
and
\[
  \text{Re}\, (\phi_sA)(i) = \frac{(ac+bd)(s^2-1) + (c^2+d^2 - a^2 - b^2)s}{(-a+cs)^2 + (-b+ds)^2}. 
\]
For an endofunction $\varphi(z)$ of $z \in \C_+$ with the representing measure $\lambda$, 
\begin{align*}
  \varphi(A.z)
  &= \text{Re}\,\varphi(i) + \int_{\overline{\R}} \phi_s(A.z)\, \lambda(ds)\\
  &= \text{Re}\,\varphi(i) + \int_{\overline{\R}} \text{Re}\, (\phi_sA)(i)\, \lambda(ds)\\
    &\quad+ (ad-bc) \int_{\overline{\R}} \frac{1+s^2}{(-a+cs)^2 + (-b+ds)^2}\, \phi_t(z)\, \lambda(ds), 
\end{align*}
whence, by the identity
\[
(-a+cs)^2 + (-b+ds)^2 = (ad-bc)^2 \frac{1+t^2}{(ct+d)^2}, 
\]
we have 
\[
    \varphi(A.z)
    = \text{Re}\,\varphi(i) + \int_{\overline{\R}} \text{Re}\, (\phi_sA)(i)\, \lambda(ds)
    + \int_{\overline{\R}} \phi_t(z)\, \lambda^A(dt) 
\]
with the measure $\lambda^A$ on $\overline{\R} = \R \sqcup \{\infty\}$ defined by 
\[
  \lambda^A(dt) = \frac{1 + (A.t)^2}{1+t^2} \frac{(ct+d)^2}{ad-bc}\, \lambda(A.dt),
\]
where $\lambda(A.dt)$ denotes the transferred measure of $\lambda(ds)$ under the change of the variable $s = A.t$.

Notice that
\[
  \frac{ds}{dt} = \frac{ad-bc}{(ct+d)^2},
  \quad
  \frac{dt}{ds} = \frac{ad-bc}{(a-cs)^2}.
\]
Notice also that
 \[
    \varphi(A.i)
    = \text{Re}\,\varphi(i) + \int_{\overline{\R}} \text{Re}\, (\phi_sA)(i)\, \lambda(ds)
    + i \int_{\overline{\R}}\lambda^A(dt),  
\]
whence
\[
      \text{Re}\,\varphi(A.i)
    = \text{Re}\,\varphi(i) + \int_{\overline{\R}} \text{Re}\, (\phi_sA)(i)\, \lambda(ds). 
\]

% Thus the following is proved. %  As a summary of the discussion so far,
Summarizing the discussions so far,

 \begin{Theorem}\label{auto}
   For $A \in \text{SL}(2,\R)$ and an endofunction $\varphi$ on $\C_+$ with its boundary measure $\lambda$ on $\overline{\R}$,  
   we have
  \[
  \varphi(A.z) = \text{Re}\,\varphi(A.i) + \int_{\overline{\R}} \frac{1+tz}{t-z}\, \lambda^A(dt)
  \quad
  (z \in \C_+).
\]
% Here $\varphi(\overline{z}) = \overline{\varphi(z)}$ by definition for $z \in \C_+$. 
 \end{Theorem}

 \begin{Remark}
   For $f \in C(\overline{\R})$,
   \[
     \int_{\overline{\R}} f(t)\, \lambda^A(dt)
     = \int_{\overline{\R}} f(A^{-1}.s) \frac{1+s^2}{1+(A^{-1}.s)^2} \frac{ad-bc}{(a-cs)^2}\, \lambda(ds).
   \]
 \end{Remark}

 Matrices for the M\"obius transformation are now relaxed to $\text{GL}_+(2,\C)$.
 To this end, we need a closer look into the semigroup $\text{GL}_+(2,\C)$.
 Though $\text{GL}_+(2,\R)$ biacts on $\text{GL}_+(2,\C)$ as a subgroup, 
 the left and right actions are not symmetrical due to the non-group character of $\text{GL}_+(2,\C)$, 
 which can be witnessed by looking at the geometric position of the boundary circline (circle or line) 
 $M.\overline{\R}$ of $M \in \text{GL}_+(2,\C)$:
 A right translate of $M$ leaves $M.\overline{\R}$ unchanged, whereas the left action moves it around.

 For $M \in \text{GL}_+(2,\C)$ with $M.\overline{\R} \subset \overline{\R} \sqcup \C_+$ a circle of radius $r$ centered at $c \in \C_+$,
 the contact degree of $M$ is defined to be $\kappa(M) = r/\text{Im}\,c$. %touching degree
 % let $\rho(M) = (\text{Im}\,c - r)/(\text{Im}\,c + r)$.
 For $M$ with $M.\overline{\R}$ parallel to $\overline{\R}$ as an extended line, we set $\kappa(M) = 1$.
 Note that $\kappa(M) = 1$ if and only if $M.\overline{\R}$ is a circline touching at one point in $\overline{\R}$.
 % As a limit case, if $M.\overline{\R}$ is a circline touching at a boundary point of $\overline{\R}$, we set $\rho(M) = 0$.
 Thus $\kappa$ is a function defined on $\text{GL}_+(2,\C) \setminus \text{GL}_+(2,\R)$ taking values in $(0,1]$,
 which turns out to be a complete invariant for orbits of the biaction of $\text{GL}_+(2,\R)$
 on the difference set $\text{GL}_+(2,\C) \setminus \text{GL}_+(2,\R)$.

 In fact, from the definition, $\kappa$ is invariant under the right action of $\text{GL}_+(2,\R)$.
 The left invariance $\kappa(AM) = \kappa(M)$ is clear for an upper triangular $A \in \text{GL}_+(2,\R)$.
 By the inversion matrix
 $J = \begin{pmatrix}
 0 & 1\\ -1 & 0
 \end{pmatrix}$, the circle $|z-c| = r$ is M\"obius transformed to a circle of radius $r/(|c|^2 - r^2)$ centered at
 $-\overline{c}/(|c|^2-r^2)$, whence $\kappa(JM) = \kappa(M)$.
 Since $\text{GL}_+(2,\R)$ is generated by upper triangle matrices and $J$, $\kappa$ is invariant under the left action
 of $\text{GL}_+(2,\R)$ as well.
 
 To see the completeness of the invariant, let $M,M' \in \text{GL}_+(2,\C)$ satisfy $\kappa(M) = \kappa(M') \equiv \kappa$. 
 If $0 < \kappa < 1$, by left translates of $M$ and $M'$ by upper triangular matrices,
 we may assume that both boundary circles have purely imaginary centers $i\gamma$ and $i\gamma'$ with radii $r$ and $r'$ respectively.
 Then $\kappa = r/\gamma = r'/\gamma'$ gives rise to a common ratio $r'/r = \gamma'/\gamma \equiv \rho >0$ and
 $M'' = \begin{pmatrix} \rho & 0\\ 0 & 1\end{pmatrix} M$ satisfies
 \[
   M'.\overline{\R} = \rho (M.\overline{\R}) = M''.\overline{\R} \iff M.\C_+ = M''.\C_+. 
 \]
 Consequently, $M^{-1}M'' \in \text{GL}(2,\C)$ gives rise to an autofunction of $\C_+$ and hence belongs to $\text{GL}_+(2,\R)$. 

 There remains the case $\kappa = 1$.
 If $M.\overline{\R}$ is a circle touching $\R$ at a real point, we may assume that it touches at $0$
 after a real affine transformation $az + b$ ($a>0$, $b \in \R$). 
 If $M.\overline{\R}$ is an extended line parallel to $\overline{\R}$,
 $(JM).\overline{\R}$ is a circle touching $\R$ at $0$.
 Thus we may assume that both boundary circles are touching $\R$ at $0$ with their centers $i\gamma$ and $i\gamma'$ respectively.
 By choosing $\rho = \gamma'/\gamma$ this time, the argument in the case $0 < \kappa < 1$ works again to conclude that
 $M^{-1}M' \in \text{GL}_+(2,\R)$. 
 
 \begin{Remark}
   $\text{SL}_+(2,\C)$ is generated by $\text{SL}(2,\R)$ and a purely imaginary shift
   $
   \begin{pmatrix}
   1 & i\\ 0 & 1
   \end{pmatrix}
   $, whereas $\text{SL}(2,\R)$ is generated by dilations $\begin{pmatrix} a & 0\\ 0 & a^{-1}\end{pmatrix}$ ($0 \not= a \in \R$),
   shifts $\begin{pmatrix} 1 & b\\ 0 & 1 \end{pmatrix}$ ($b \in \R$) and the inversion $J$. 
 \end{Remark}

 Given an endomatrix $M \in \text{GL}_+(2,\C)$, the associated endofunction $M.z$ of $z \in \C_+$ is denoted by
 $\phi_M(z)$ with the representing measure of $\phi_M$ by $\mu_M$. Thus an atomic endofunction $\phi_s$ is of the form $\phi_M$ for
 \[
   M =
 \frac{1}{\sqrt{1+s^2}} 
 \begin{pmatrix}
   s & 1\\
   -1 & s
 \end{pmatrix}
 \doteq
 \begin{pmatrix}
   s & 1\\
   -1 & s
 \end{pmatrix}. 
\]
Here $\doteq$ indicates the proportionality of matrices. 
Notice that $M$ takes the form
$
 \begin{pmatrix}
   \sin\theta & \cos\theta\\
   -\cos\theta & \sin\theta
 \end{pmatrix}
$ in terms of $\theta = \arctan s$ ($-\pi/2 \leq \theta \leq \pi/2$). 
 
\begin{Lemma}\label{unbounded} % touching at infinity 
  For an endofunction\footnote{We reset the usage of symbols $\phi$ and $\mu$ so that these indicate ones on $\C_+$ instead of $D$.}
  $\phi = \phi_M$ with $M \in \text{GL}_+(2,\C)$, 
  % For a holomorphic function $\varphi$ of M\"obius transform on $\C_+$,
  the following conditions are equivalent.
   % For an endofunction $\varphi$ of M\"obius transform, the following conditions are equivalent. 
   \begin{enumerate}
   \item
   The representing measure $\mu$ of $\phi$ contains an atomic measure.
   \item
     The image $\phi(\C_+)$ is an unbounded subset of $\C_+$.
   \item
   The image $\phi(\C_+)$ is $\C_+ + ir$ with $r \geq 0$.
   \item
     % The endofunction $\phi$ is of the form
     $\phi = a \phi_s + b$ with $a > 0$, $b \in \R \sqcup \C_+$ and $s \in \overline{\R}$, where $s$ is characterized by
     $\phi(s) = \infty$. 
   \end{enumerate} 
 \end{Lemma}

 \begin{proof}
   (i) $\Longrightarrow$ (ii): Assume that $\phi(\C_+)$ is bounded.
   The endofunction $\phi$ is then continuously extended to $\overline{\C_+} = \C_+ \sqcup \overline{\R}$ as a linear fractional function.
   By Corollary~\ref{mass}, this means that $\mu$ contains no atomic mass.

   (ii) $\Longrightarrow$ (iv): Since $\phi(\C_+)$ is unbounded and the boundary $\phi(\overline{\R})$ is a circline in
   the Riemann sphere $\overline{\C}$ as a property of M\"obius transform, the path $\phi(t)$ ($t \in \overline{\R}$) passes through
   $\infty$ and we can find a point $s \in \overline{\R}$ satisfying $\phi(s) = \infty$. Thus we have an expression
   \[
     \phi(z) =
     \begin{cases}
       \frac{a}{s-z} + b &(s \in \R),\\
       az + b &(s = \infty)
     \end{cases}
   \]
   with $0 \not= a \in \C$ and $b \in \C$. In view of the positivity $\phi(\C_+) \subset \C_+$, the behavior of $\phi(z)$
   near $z=s$ or $z = \infty$ and then near $z = \infty$ or $z=0$ imply $a>0$ and $\text{Im}\,b \geq 0$.
   Consequently, if $s \in \R$, 
   \[
     \phi(z) = \frac{a}{1+s^2} \phi_s(z) + \frac{as}{1+s^2} + b
   \]
   is the desired form, whereas 
   $\phi(z) = a\phi_s(z) + b$ for $s = \infty$.

   (iv) $\Longrightarrow$ (iii): $\phi(\C_+) = a \phi_s(\C_+) + b = \C_+ + b = \C_+ + ir$ for $r = \text{Im}\, b$.

   (iii) $\Longrightarrow$ (ii) is obvious.

   (iv) $\Longrightarrow$ (i): For $b \in \R$, $a\phi_s + b$ is an autofunction with $\mu$ an atomic measure of mass $a$.
   For $b \in \C_+$, $\phi$ is a semilinear combination of endofunctions $\phi_s$ and $b$, whence again $\mu$ contains an atomic measure of
   mass $a$. 
 \end{proof}

 In what follows, we regard the representing measure $\mu_M$ of $\phi_M$ as a Radon measure on $\overline{\R}$
 and the associated linear functional on $C(\overline{\R})$ is also denoted by $\mu_M$, i.e., 
 for $f \in C(\overline{\R})$, 
 \[
   \mu_M(f) = \int_{\overline{\R}} f(s)\, \mu_M(ds).
 \]
 
 We claim that $\mu_M(f)$ is continuous in $M \in \text{GL}_+(2,\C)$.
 As a preliminary fact, observe that the total mass $\mu_M(\overline{\R}) = \text{Im}\,\phi_M(i)$ (the $L^1$ or the functional norm of $\mu_M$)
 as well as the real constant term $\text{Re}\,\phi_M(i)$ of $\phi_M$ are continuous functions
 of $M \in \text{GL}_+(2,\C)$. % because their complex combination $\phi_M(i)$ is continuous in $M$. 

 For $M = \begin{pmatrix} a & b\\ c & d\end{pmatrix} \in \text{GL}_+(2,\C)$, $M.\C_+$ is bounded exactly when $c \not= 0$ and
 $d/c \not\in \R$. In that case, there exists a neighborhood $\sN$ of $M$ in $\text{GL}_+(2,\C)$ such that
 \begin{enumerate}
   \item
     $\phi_N(z)$ is continuously extended to $\overline{\R} \sqcup \C_+$ for each $N \in \sN$ and
   \item
     $\phi_N(x)$ is continuous as a function of  $(x,N) \in \overline{\R}\times \sN$.
     % is uniformly bounded for $N \in \sN$ and continuous on $N \in \sN$.
 \end{enumerate}
 Then $\mu_N$ for $N \in \sN$ is supported by $\R$ and equivalent to the Lebesgue measure on $\R$ so that
 $\mu_N(dx) = \bigl(\text{Im}\, \phi_N(x)/\pi(1+x^2)\bigr)\, dx$, whence the continuity in (ii) implies that
 $\sN \ni N \mapsto \mu_N$ is continuous with respect to the $L^1$-norm on $C(\overline{\R})^*$.

 Next assume that $M.\C_+$ is unbounded, i.e., $c=0$ or $d/c \in \R$. We focus on the case $s \equiv -d/c \in \R$ with $c \not= 0$
 as the case $c = 0$ can be dealt with in a similar fashion. Then the expression
 \[
   \phi_M(z) = -\frac{as+b}{(1+s^2)c} \phi_s(z) + \frac{a-bs}{(1+s^2)c}
 \]
shows that $(as+b)/c < 0$ and $\text{Im} (a-bs)/c \geq 0$, which is compatible with the formula
 \[
   (1+s^2) \mu_M(\{s\}) = \lim_{y \to +0} y\,\text{Im}\, \phi_M(s+iy) = - \frac{as+b}{c}
 \]
 in Corollary~\ref{mass}.
 
 For any $f \in C(\overline{\R})$, we shall check the continuity of $\mu_{M'}(f)$ ($M' \in \text{GL}_+(2,\C)$) at $M' = M$, i.e., 
 given a matrix sequence $M_n = \begin{pmatrix} a_n & b_n\\ c_n & d_n \end{pmatrix}$ in $\text{GL}_+(2,\C)$ converging to $M$,
 $\lim_n \mu_{M_n}(f) = \mu_M(f)$.
 % , i.e., $a_n \to a$, $b_n \to b$, $c_n \to c$ and $d_n \to d$.
 To simplify the notation, we drop off $M$ and write as $\phi$ and $\mu_n$ instead of $\phi_M$ and $\mu_{M_n}$. % for example.

 Since $c_n \to c \not= 0$, we may suppose that $c_n \not= 0$ for all $n$. Notice that the measure $\mu_n$ is supported by $\R$ then. 
 For $n$ satisfying $d_n/c_n \in \R$, we write $s_n = -d_n/c_n$. In view of $d_n/c_n \to d/c$, the subsequence $(s_n)$ of $(-d_n/c_n)$
 converges to $s$ if there are infinitely many $s_n$'s.
 With these observations in mind, we shall go into estimates relevant to the convergence.

 Given $\epsilon>0$, choose $\delta>0$ large enough so that $|f(x) - f(s)| \leq \epsilon$ ($|x-s| \leq \delta$) and then
 choose $m$ large enough so that $|s_n-s| \leq \delta/2$ ($n \geq m$).
 Since $|\phi(x)|$ ($|x-s| \geq \delta$) is bounded and $\lim_{m \leq n \to \infty} \phi_n(x) = \phi(x)$ uniformly on $|x-s|\geq \delta$,
 we can find $m' \geq m$ such that
 \[
   |\phi_n(x) - \phi(x)| \leq \epsilon
   \quad
   (|x-s| \geq \delta, n \geq m'), 
 \]
and then 
 \begin{align*}
   &\left| \int_{|x-s| \geq \delta} f(x)\, \mu_n(dx) - \int_{|x-s| \geq \delta} f(x)\, \mu(dx) \right|\\
   &\qquad= \left| \int_{|x-s| \geq \delta} f(x) \frac{\text{Im}\, \phi_n(x) - \text{Im}\,\phi(x)}{\pi(1+x^2)}\, dx \right|\\
   &\qquad\leq \epsilon \| f\| \int_{|x-s| \geq \delta} \frac{1}{\pi(1+x^2)}\, dx \leq \epsilon \| f \|. 
 \end{align*}
 Notice here that the initial estimate on $f$ is automatically satisfied for the choice $f \equiv 1$ and the above inequality takes the form 
 \[
   \bigl|\mu_n(|x-s| \geq \delta) - \mu(|x-s| \geq \delta)\bigr| \leq \epsilon
   \quad (n \geq m'). 
 \]
 As for the estimate on the $\delta$-neighborhood of $s$, we have
  \begin{align*}
   &\left| \int_{|x-s| \leq \delta} f(x)\, \mu_n(dx) - \int_{|x-s| \leq \delta} f(x)\, \mu(dx) \right|\\
   &\qquad= \int_{|x-s| \leq \delta} |f(x) - f(s)|\, \mu_n(dx) + \int_{|x-s| \leq \delta} |f(x) - f(s)|\, \mu(dx)\\
   &\qquad\qquad+ |f(s)| \bigl|\mu_n(|x-s| \leq \delta) - \mu(|x-s| \leq \delta)\bigr|\\
  % &\qquad\leq \epsilon (\mu_n(\R) + \mu(\R)) + \| f\| \bigl|\mu_n(|x-s| \leq \delta) - \mu(|x-s| \leq \delta)\bigr|\\
  &\qquad\leq \epsilon (\mu_n(\R) + \mu(\R)) + \| f\|\, \bigl|\mu_n(|x-s| \leq \delta) - \mu(|x-s| \leq \delta)\bigr|. 
 \end{align*}
 In the last term of measure difference,
 \begin{align*}
   &\bigl| \mu_n(|x-s| \leq \delta) - \mu(|x-s| \leq \delta) \bigr|\\
   &\qquad= \bigl| \mu_n(\R) - \mu_n(|x-s| \geq \delta) - \mu(\R) + \mu(|x-s| \geq \delta) \bigr|\\
   &\qquad\leq |\mu_n(\R) - \mu(\R)| + \bigl| \mu(|x-s| \geq \delta) - \mu_n(|x-s| \geq \delta) \bigr| 
 \end{align*}
 is used to have
 \[
   \left| \int_{|x-s| \leq \delta} f(x)\, \mu_n(dx) - \int_{|x-s| \leq \delta} f(x)\, \mu(dx) \right|
   \leq |\mu_n(\R) - \mu(\R)| + \epsilon. 
 \]
 
 Putting all these together,
 \[
   | \mu_n(f) - \mu(f)| \leq 2\epsilon \| f\| + \epsilon (\mu_n(\R) + \mu(\R))
      \quad (n \geq m'). 
 \]
 Since $\mu_n(\R) = \text{Im}\, \phi_n(i) \to \text{Im}\,\phi(i) = \mu(\R)$, this means that $\lim_n \mu_n(f) = \mu(f)$.

 \begin{Theorem}\label{continuity}
   The Radon measure $\mu_M$ on $\overline{\R}$ is weak*-continuous in $M \in\text{GL}_+(2,\C)$ and,
   if we restrict $M$ to an open subset consisting of $M$ for which $M.\C_+$ is bounded, it is continuous with respect to
   the functional norm of $C(\overline{\R})^*$. 
 \end{Theorem}

 \begin{Remark}
 The norm continuity breaks at $M$ for which $M.\C_+$ is unbounded. % pushing limit. 
 \end{Remark}
 
 % We separate two cases: the bounded case for which $M_n.\C_+$ is bounded ($d_n/c_n \not\in \R$) and
 % the unbounded case for which  $M_n.\C_+$ is unbounded ($d_n/c_n \in \R$).
 
 % \[
 %   \lim_n  \mu_n(\overline{\R}) = \lim_n \text{Im}\, \phi_n(i) = \text{Im}\,\phi(i) = \mu(\overline{\R}).
 % \]
 
 % Bounded case: $\mu_n$ contains no atomic measure and 
 
 % $s_n = -d_n/c_n$. 
 % \[
 %   \lim_n  \mu_n(\{s_n\}) = -\lim_n \frac{a_ns_n+b_n}{(1+s_n^2)c_n} = -\frac{as+b}{(1+s^2)c} = \mu(\{s\}). 
 % \]

 % When $c=0$, $a/d > 0$ and $\text{Im}\, (b/d) \geq 0$. 

 % Let $\text{SL}_+(2,\C) = \{ A \in \text{SL}(2,\C); \text{$A.z$ is an endofunction of $z \in \C_+$}\}$, which
 % is a subsemigroup of $\text{SL}(2,\C)$ containing $\text{SL}(2,\R)$ and generated as a semigroup by
 % dilations, real shifts, the inversion and a purely imaginary shift
 % \[
 %   \begin{pmatrix}
 %     1 & i\\
 %     0 & 1
 %   \end{pmatrix}. 
 % \]
 
 % For $A =
 % \begin{pmatrix}
 %   a & b\\
 %   c & d
 % \end{pmatrix}$ in $\text{SL}_+(2,\C)$,

 For $M \in \text{GL}_+(2,\C)$ and $s \in \R$, let
 \[
   {}_sM = \frac{1}{\sqrt{1+s^2}}\begin{pmatrix} s & 1\\ -1 & s\end{pmatrix} M
 \]
 be a left translate of $M$ in $\text{GL}_+(2,\C)$ with the limiting matrices defined by ${}_{\pm\infty}M = \pm M$
 % write ${}_sM = \begin{pmatrix} s & 1\\ -1 & s\end{pmatrix} M$
 and $\mu_{{}_sM}$ be the representing measure of $\phi_{{}_sM}$. When $M$ is obvious, $\mu_{{}_sM}$ is simply denoted by
 $\mu_s$. Notice that both $\phi_{{}_sM}$ and $\mu_{{}_sM}$ are well-defined as functions of $s \in \overline{\R}$
 (especially at $s = \infty$) because all these depend on $M$ through the projective ray $\C^\times M$. 

 By parametric continuity, this family of measures gives rise to an operator $\Lambda_M$ 
 on the Banach space $C(\overline{\R})$ by
 \[
   \Lambda_M(f): \overline{\R} \ni s \mapsto \mu_{{}_sM}(f) \in \C, % = (\mu_{{}_sM}(f))_{s \in \overline{\R}},
 \]
 which preserves positivity in the sense that $\Lambda_M(f) \geq 0$ for $f \geq 0$.
 Note that $\Lambda_M$ preserves the unit function $1$ in $C(\overline{\R})$ (a so-called Markov operator) if and only if
 $\text{Im}\, \phi_s(M.i) = \mu_{{}_sM}(1) = 1$ ($s \in \overline{\R}$), i.e., $M.i = i$. 

 In the integral representation
 \[
   \varphi(z) = \text{Re}\,\varphi(i) + \int_{\overline{\R}} \phi_s(z)\, \lambda(ds)
 \]
 of an endofunction $\varphi$, $z$ is replaced by $M.z$ to have
 \begin{align*}
   \varphi(M.z) &= \text{Re}\,\varphi(i) + \int_{\overline{\R}} \phi_{{}_sM}(z)\, \lambda(ds)\\
                &= \text{Re}\,\varphi(i) + \int_{\overline{\R}} 
                  \left(\text{Re}\,\phi_s(M.i) + \int_{\overline{\R}} \phi_t(z)\, \mu_{{}_sM}(dt)\right) \lambda(ds)\\
                &= \text{Re}\,\varphi(i) + \int_{\overline{\R}} \text{Re}\,\phi_s(M.i)\,\lambda(ds)
                  + \int_{\overline{\R}} \lambda(ds)\,\int_{\overline{\R}} \phi_t(z)\, \mu_{{}_sM}(dt). 
 \end{align*}
 To interpret the last term, consider the repeated integral 
 \[
   \int_{\overline{\R}} \lambda(ds)\,\int_{\overline{\R}} f(t)\, \mu_{{}_sM}(dt)
   = \int_{\overline{\R}} \lambda(ds)\,\mu_{{}_sM}(f)
 \]
 of $f \in C(\overline{\R})$. Since $\mu_{{}_sM}(f)$ is continuous in $s \in \overline{\R}$ by Theorem~\ref{continuity}
 and $\mu_{{}_sM}(f) \geq 0$ for $f \geq 0$,
 this gives a positive linear functional of $f \in C(\overline{\R})$ and therefore defines a Radon measure $\lambda^M$ on $\overline{\R}$ by
 \[
   \int_{\overline{\R}} f(t)\, \lambda^M(dt) =  \int_{\overline{\R}} \lambda(ds)\,\int_{\overline{\R}} f(t)\, \mu_{{}_sM}(dt). 
 \]

 Discussions so far are now organized as follows:
 
 \begin{Theorem}
   Given an endomatrix $M$ and an endofunction $\varphi$ with $\lambda$ the representing measure of $\varphi$,
   the M\"obius-transformed endofunction is represented by 
 \[
   \varphi(M.z) = \text{Re}\,\varphi(i) + \int_{\overline{\R}} \text{Re}\,\phi_s(M.i)\,\lambda(ds) + \int_{\overline{\R}} \phi_t(z)\, \lambda^M(dt).
 \]
 In other words,
 \[
   \text{Re}\, \varphi(M.i) = \text{Re}\,\varphi(i) + \int_{\overline{\R}}  \text{Re}\,\phi_s(M.i)\,\lambda(ds)
 \]
 and the representing measure of $\varphi(M.z)$ is given by $\lambda^M$.
\end{Theorem}

\begin{Corollary}
For endomatrices $M$ and $N$, we have $\Lambda_{MN} = \Lambda_M \Lambda_N$. 
\end{Corollary}

\begin{proof}
  For the choice $\varphi = \phi_{{}_sM}$ and $\varphi(N.z)$, the theorem gives
  \[
    \mu_{{}_sMN} = \int \mu_{{}_sM}(dt) \mu_{{}_tN}, 
  \]
  which is utilized to see
  \begin{align*}
    (\Lambda_{MN}f)(s) = \mu_{{}_sMN}(f) &= \int \mu_{{}_sM}(dt) \mu_{{}_tN}(f)\\
                                         &= \int (\Lambda_Nf)(t)\, \mu_{{}_sM}(dt) = (\Lambda_M\Lambda_Nf)(s).
  \end{align*}
\end{proof}

\begin{Remark}
The measure support of $\lambda^M$ is $M^{-1}.[\lambda]$ if $M$ is an automatrix and $\overline{\R}$ otherwise. 
\end{Remark}

We shall give the parametrized measure $\mu_{{}_sM}$ more explicitly in three cases.

 The case that $\phi_M$ is an autofunction, i.e., $M \in \text{GL}_+(2,\R)$, is already dealt with in Theorem~\ref{auto};
 $\mu_{{}_sM}$ is an atomic measure at $M^{-1}.s \in \overline{\R}$ with the mass given by
 \[
   \frac{1+s^2}{1+(M^{-1}.s)^2} \frac{ad-bc}{(a-cs)^2}
   \quad
   \text{for}\ 
   M =
   \begin{pmatrix}
     a & b\\
     c & d
   \end{pmatrix}
   \in \text{GL}_+(2,\R).
 \]
 % \[
 %     \int_{\overline{\R}} f(t)\, \mu_{{}_sM}(dt)
 %     = \int_{\overline{\R}} f(A^{-1}.s) \frac{1+s^2}{1+(A^{-1}.s)^2} \frac{ad-bc}{(a-cs)^2}\, \lambda(ds).
 % \]

 When $M.\overline{\R}$ is not contacting $\overline{\R}$, i.e., $0 < \kappa(M) < 1$,
 $\phi_{{}_sM}(z)$ is continuously extended to $\overline{\R}\sqcup \C_+$
 with the measure $\mu_{{}_sM}$ supported by $\R$ in such a way
 that it is equivalent to the Lebesgue measure on $\R$ with the density function given by a rational function
 $\text{Im}\,\phi_{{}_sM}(x)/\pi(1+x^2)$ of $x \in \R$, i.e., 
 \[
   \mu_{{}_sM}(dx) = \frac{\text{Im}\,\phi_{{}_sM}(x)}{\pi(1+x^2)} dx. 
 \]
 
 As a mixture of these, there is the contact case, i.e., $M \in \text{GL}_+(2,\C) \setminus \text{GL}_+(2,\R)$ fulfilling $\kappa(M) = 1$.
 If ${}_sM.\overline{\R}$ is a (bounded) circle, the description of non-contact case remains valid. 
 Otherwise, % , given $s \in \overline{\R}$,
 the boundary circline ${}_sM.\overline{\R}$ is unbounded exactly at one point $s$ in $\overline{\R}$ and 
 there exists $t \in \overline{\R}$ satisfying ${}_sM.t = \infty \iff s = M.t$, 
 which enables us to apply Lemma~\ref{unbounded} for ${}_sM$ to obtain
 an expression $\phi_{{}_sM} = p \phi_t + q$ with $p > 0$ and $\text{Im}\,q \geq 0$.
 Consequently we have 
\[
  {}_sM \doteq
  \begin{pmatrix}
    p & q\\
    0 & 1
  \end{pmatrix}
  % \frac{1}{\sqrt{1+t^2}}
  \begin{pmatrix}
    t & 1\\
    -1 & t
  \end{pmatrix}
\iff 
  \begin{pmatrix}
    p & q\\
    0 & 1
  \end{pmatrix}
  \doteq
  \begin{pmatrix}
    s & 1\\
    -1 & s
  \end{pmatrix}
  M
   \begin{pmatrix}
    t & -1\\
    1 & t
  \end{pmatrix}. 
\]
From the last relation, one sees that
\[
  p = \frac{(ad-bc)(1+s^2)}{(-a+cs)^2 + (-b+ds)^2},
  \quad
  q = \frac{(ac+bd)(s^2-1) + (c^2+d^2-a^2-b^2)s}{(-a+cs)^2 + (-b+ds)^2}
\]
and
\[
  \mu_{{}_sM}(dx) = p \delta(x-t)\, dx + \frac{\text{Im}\,q}{\pi(1+x^2)}\, dx.
\]

% For other $s \in \overline{\R}$, ${}_{s}M.\overline{\R}$ is bounded

\begin{Example}
  Given $\sigma,\tau \in \R$ and $p>0,r>0$,
  \[
    M(\sigma) = %\frac{1}{\sqrt{1+s^2}\sqrt{1+t^2}}
    \begin{pmatrix}
      \sigma & -1\\
      1 & \sigma
    \end{pmatrix}
    \begin{pmatrix}
      p & ir\\
      0 & 1
    \end{pmatrix}
    \begin{pmatrix}
      \tau & 1\\
      -1 & \tau
    \end{pmatrix}
  \]
  is an endomatrix which conveys $\C_+$ onto a disk of radius $(1+\sigma^2)/2r$ centered at $\sigma + i(1+\sigma^2)/2r$
  in such a way that
  \[
    {}_sM(\sigma) = \frac{s-\sigma}{\sqrt{1+s^2}} M\left(\frac{1+s\sigma}{s-\sigma}\right)
    \quad
    (s \not= \sigma)
  \]
  and
  \[
    {}_\sigma M(\sigma) = \sqrt{1+\sigma^2}  \begin{pmatrix}
      p & ir\\
      0 & 1
    \end{pmatrix}
    \begin{pmatrix}
      \tau & 1\\
      -1 & \tau
    \end{pmatrix}. 
  \]
  Thus ${}_sM(\sigma).\C_+$ is a bounded disk contacting $\R$ at $s$ for $s \not= \sigma$ and
  $\C_+ + ir$ for $s = \sigma$. 
\end{Example}

%  Let $M \in \text{GL}_+(2,\C)$ and assume that $M.t = s$ with $s,t \in \overline{\R}$.
%  Letting $M' = \frac{1}{\sqrt{1+s^2}} \begin{pmatrix} s & 1\\ -1 & s\end{pmatrix} M \in \text{GL}_+(2,\C)$, we then have
%  \[
%    % \left(
%    % \begin{pmatrix}
%    %   s & 1\\
%    %   -1 & s
%    % \end{pmatrix}
%    % M\right).t
%  M'.t 
%  = \frac{1}{\sqrt{1+s^2}}
%  \begin{pmatrix}
%      s & 1\\
%      -1 & s
%    \end{pmatrix}.s = \infty
% \]
% and Lemma~\ref{unbounded} can be applied to obtain
% an expression $\phi_{M'} = a \phi_u + b$ with $a > 0$, $u \in \overline{\R}$
% and $\text{Im}\,b \geq 0$. In view of $\phi_{M'}(t) = \infty$, $u$ must be equal to $t$ and hence 
% \[
%   \frac{1}{\sqrt{1+s^2}}
%   \begin{pmatrix}
%      s & 1\\
%      -1 & s
%   \end{pmatrix}
%   M \doteq
%   \begin{pmatrix}
%     a & b\\
%     0 & 1
%   \end{pmatrix}
%   \frac{1}{\sqrt{1+t^2}}
%   \begin{pmatrix}
%     t & 1\\
%     -1 & t
%   \end{pmatrix}
% \]

% \[
%   M = \pm \sqrt{\det(M)}  \frac{1}{\sqrt{1+s^2}}
%   \begin{pmatrix}
%      s & -1\\
%      1 & s
%   \end{pmatrix}
%   \begin{pmatrix}
%     a^{1/2} & a^{-1/2}b\\
%     0 & a^{-1/2}
%   \end{pmatrix}
%     \frac{1}{\sqrt{1+t^2}}
%   \begin{pmatrix}
%     t & 1\\
%     -1 & t
%   \end{pmatrix}
% \]

%  Then
%  \[
%    M.z - s = \frac{az+b}{cz+d} - \frac{at+b}{ct + d} = \frac{(ad-bc)(z-t)}{(ct+d)(cz+d)} 
%  \]
%  and
%  \[
%    -\frac{1}{M.z - s} = \frac{ct+d}{ad-bc} \frac{cz+d}{t-z}
%  \]

 \section{Localized Positivity}
 Recall that an endofunction $\varphi$ on $\C_+$ is characterized as a holomorphic function
 different from a real constant and satisfying boundary positivity. 
We shall here modify the condition to localized one near the boundary support $[\varphi]$.
When $[\varphi] = \overline{\R}$, this is just the global positivity in view of the maximum principle for the harmonic
function $\text{Im}\,\varphi$. So assume that $\overline{\R} \setminus [\varphi] \not= \emptyset$ and
$\text{Im}\,\varphi \geq 0$ on the intersection of $\C_+$ with a neighborhood of $[\varphi]$ in $\overline{\C}$.
Then, for each $r>0$, $\varphi_r(z) = \varphi(z) + ir$ satisfies $\text{Im}\, \varphi_r \geq 0$ on $\C_+ \cap U$ with
$U$ a sufficiently small neighborhood of $\overline{\R}$ in $\overline{C}$, whence $\varphi_r$ is an endofunction.
Consequently $\text{Im}\,\varphi_r(z) \geq 0$ ($z \in \C_+$) and the boundary positivity holds for $\varphi$ as a limit $r \downarrow 0$.
Since $\varphi$ is not a real constant by assumption, $\varphi$ is an endofunction. 

\begin{Theorem}
  A holomorphic function $\varphi$ on $\C_+$ different from real constants is an endofunction if and only if 
  there exists a neighborhood $U$ of $[\varphi]$ in $\overline{\C}$
  satisfying $\text{Im}\, \varphi(z) \geq 0$ ($z \in U \cap \C_+$). 
\end{Theorem}

\begin{Remark}
This fact is realized for a certain rational combination of power functions in \cite[Theorem 1.1]{NW}.
\end{Remark}

% \begin{Remark}
%   This fact can be simply understood as follows: Under the assumption of the theorem, $\varphi_r(z) = \varphi(z) + ir$ ($r>0$)
%   takes strictly positive values as its imaginary part near the boundary of $\C_+$,
%   whence $\text{Im}\,\varphi_r \geq 0$ by the maximum principle for the harmonic
%   function $\text{Im}\,\varphi_r$ and then $\text{Im}\,\varphi(z) = \lim_{r \to +0} \text{Im}\,\varphi_r(z) \geq 0$. 
% \end{Remark}

% \begin{Theorem}
%   For a real holomorphic function $\varphi$ on $\C \setminus \R$, if there is a neighborhood $U$ of $[\varphi]$ in $\overline{\C}$
%   satisfying $\text{Im}\, \varphi(z) \geq 0$ ($z \in U \cap \C_+$),
%   we can find $\alpha \in \R$ and a finite positive measure $\lambda$ on $\overline{\R}$ so that
%   \[
%     \varphi(z) = \alpha + \int_{\overline{\R}} \frac{1+sz}{s-z}\, \lambda(ds) 
%     \quad
%     (z \in \C_+). 
%   \] 
% \end{Theorem}

\begin{Example} Real rational endofunctions are described in \cite{Do, Na}.
  Let us reexamine it in our context.
  Given a rational function $f$,
  we seek for the condition
  which makes the restriction $\varphi = f|_{\C_+}$ be an endofunction.
  Since $\varphi$ is holomorphic and non-vanishing on $\C_+$, poles and zeros of $f$ are out of $\C_+$.
  
  When the boundary support $[\varphi]$ is not full, $f(z) \in \R$ for $z \in \R \setminus [\varphi]$, which 
  implies that $f$ is a real rational function. Therefore all poles and zeros are concentrated in
  the boundary $\overline{\R}$ with $[\varphi]$ consisting of poles.
  Then, in view of the boundary behavior of endofunctions (Proposition~\ref{Stoltz}),
  poles of $f$ must be simple.
  The partial fraction expansion of $f$ therefore takes the form
  \[
    f(z) = az + b + \sum_j \frac{c_j}{s_j-z}
  \]
  with $a,b,c_j$ and $s_j \in \R$.
  By Corollary~\ref{mass}, $a \geq 0$ and $c_j > 0$ to have a positive boundary measure.
  Consequently, under this condition,
  the choice $\lambda_0(ds) = a \delta_\infty + \sum_j c_j \delta_{s_j}$
  gives the expression
  \[
    \varphi(z) = b + \int_{\R} \frac{1+s_jz}{s_j - z}\, \lambda_0(ds),
  \]
  showing that $\varphi$ is an endofunction.

  For non-real rational function $f$, the boundary support is $\overline{\R}$
  and real poles (including $\infty$) need to behave in the same way
  to obtain an endofunction. Other poles should be in $\C_-$ and
  $\psi = \varphi - az - (b+ic) - \sum_j c_j/(s_j-z)$ with $c \in \R$
  is continuously extended to $\overline{\R}$
  so that it vanishes at $\infty$. Thus $\varphi$ is an endofunction if and only if
  \[
    c \geq -\min\{ \text{Im}\, \psi(x); x \in \R\}.
  \]
  In fact, if $c < -\min$, there is an interval $I$ such that $c + \text{Im}\,\psi(x) < 0$ ($x \in I$),
  whence we can find $x \in I \setminus [\lambda_0]$ so that $\text{Im}\,\varphi(x+iy) < 0$
  for sufficiently small $y > 0$.
  If $c > -\min$, the rational function $\psi +ic$ satisfies boundary positivity and
  it is an endofunction. The same holds for $c = -\min$ as a pointwise limit of endofunctions. 
  
  Moreover, if this is the case, the representing measure of $\varphi$ is given by 
  \[
    \lambda_0(ds) + \frac{c + \text{Im}\,\psi(s)}{\pi(1+s^2)}\, ds. 
  \]
% The integral representation of $ir$. 
\end{Example}

\appendix
\section{}
To extract more information on point masses of $\lambda$,
we introduce the Stoltz sector of apex $0$ and aperture $0 < \phi < \pi/2$ by 
\[
  V_\phi = \{ x +iy \in \C; |x| < y \tan\phi\}, 
\]
which is an open convex cone in $\C_+$ and has a spindle shape in $\overline{\C}$.
Given $\alpha \in \R$ and a signed measure $\lambda$ on $\overline{\R}$, let $\varphi$ be a holomorphic function on $\C \setminus \R$
defined by
\[
  \varphi(z) = \alpha + \int_{\overline{\R}} \frac{1+sz}{s-z}\, \lambda(ds)
  = \alpha + \lambda(\{\infty\})z + \int_\R \frac{1+sz}{s-z}\, \lambda(ds). 
\]
Notice that $\varphi(i) = \alpha + i\lambda(\overline{\R})$.

\begin{Proposition}\label{Stoltz}
% Given $b>0$, $c \in \R$ and $0 < \phi < \pi/2$, the following holds. 
  Given $b>0$ and $c \in \R$, both
  \[
    \{ \varphi(x+iy)/y; x+iy \in V_\phi, y \geq b \}
  \]
  and
  \[
    \{ y \varphi(x+iy); x + iy \in c + V_\phi, 0 < y \leq b\}
  \]
  are bounded sets and we have 
    \[
     \lim_{V_\phi \ni x+iy \to \infty} \frac{\varphi(x+iy) - \lambda(\{\infty\})(x+iy)}{y} = 0, 
   \]
      \[
       \lim_{c+V_\phi \ni x+iy \to c} y \left(\varphi(x+iy) + \lambda(\{c\}) \frac{1+c^2}{x+iy-c}\right) = 0.
   \]
\end{Proposition}

\begin{proof}
  % In the real part of $\varphi(x+iy)/y$, i.e.,
  In the expression
  \[
  \frac{\text{Re}\,\varphi(x+iy) - \lambda(\{\infty\})x}{y}
  = \frac{\text{Re}\,\varphi(i)}{y} + \int_\R \frac{(s-x)(1+sx) - sy^2}{y((s-x)^2 + y^2)}\, \lambda(ds), 
\]
we regard the integrand as a function of $(s,y)$ with $x$ a parameter and rewrite it by the polar coordinates
$s-x = \rho\cos\theta$, $y = \rho\sin\theta$ ($\rho>0$, $0 < \theta <\pi$) to have an expression
\[
  \frac{1+x^2-y^2}{\rho y} \cos\theta + \frac{x}{y} \cos(2\theta)
\]
with its modulus estimated by 
\[
  \frac{|1+x^2-y^2|}{\rho y} + \frac{|x|}{y} % \leq \frac{|1+x^2-y^2|}{y^2} + \frac{|x|}{y}
  \leq 1 + \frac{1+x^2}{y^2} + \frac{|x|}{y}
  \leq 1 + \frac{1}{b^2} + \tan\phi + \tan^2\phi
 % \leq \left(1 + \frac{1}{\epsilon^2}\right)^2
\]
for $x+iy \in V_\phi$ fulfilling $y \geq b$.

The bounded convergence theorem is now applied to see that, for a sequence $(x_n+iy_n)$ in $V_\phi$ satisfying
$y_n \to \infty$,
\begin{multline*}
 \lim_{n \to \infty} \int_\R \frac{(s-x_n)(1+sx_n) - sy_n^2}{y_n((s-x_n)^2 + y_n^2)}\, \lambda(ds)\\
 = \int_\R \lim_{n \to \infty}\frac{(s-x_n)(1+sx_n) - sy_n^2}{y_n((s-x_n)^2 + y_n^2)}\, \lambda(ds) = 0,
\end{multline*}
which implies that, when $y>0$ is large, $y^2 \leq x^2 + y^2 \leq y^2/\cos^2\phi$ is large as well for $x+iy \in V_\phi$
and we can make 
\[
  \int_\R \frac{(s-x)(1+sx) - sy^2}{y((s-x)^2 + y^2)}\, \lambda(ds)
\]
arbitrarily small.

As to the imaginary part, in the expression 
\[
  \frac{\text{Im}\, \varphi(x+iy) - \lambda(\{\infty\})y}{y} = \int_\R \frac{1+s^2}{(s-x)^2 + y^2}\, \lambda(ds), 
\]
its modulus is estimated by 
\[
  \left|  \int_\R \frac{1+s^2}{(s-x)^2 + y^2}\, \lambda(ds) \right|
  \leq \int_\R \frac{1+s^2}{(s-x)^2 + y^2}\, |\lambda|(ds). 
\]
Since the range of the last integrand
\[
  \displaystyle r = \frac{1+s^2}{(s-x)^2 + y^2} \iff r((s-x)^2 + y^2) = 1 + s^2
\]
is ruled by 
$r^2y^2 - r(x^2+y^2+1) + 1 \leq 0$,
% i.e., 
% \[  \frac{x^2+y^2+1}{2y^2} - \sqrt{\left(\frac{x^2+y^2+1}{2y^2}\right)^2 - \frac{1}{y^2}} \leq r \leq
%   \frac{x^2+y^2+1}{2y^2} + \sqrt{\left(\frac{x^2+y^2+1}{2y^2}\right)^2 - \frac{1}{y^2}}, 
% \]
we obtain an estimate
\begin{align*}
  % \max_{s \in \overline{\R}} \frac{1+s^2}{(s-x)^2+y^2} =
  \sup_{s \in \R} \frac{1+s^2}{(s-x)^2+y^2}
  &= \frac{x^2+y^2+1}{2y^2} + \sqrt{\left(\frac{x^2+y^2+1}{2y^2}\right)^2 - \frac{1}{y^2}}\\
  &\leq \frac{x^2+y^2+1}{y^2}
    = 1 + \frac{1+x^2}{y^2} \leq 1 + \frac{1}{b^2} + \tan^2\phi. 
\end{align*}
Again the bounded convergence theorem is applied to have
\[
   \lim_{V_\phi \ni x+iy \to \infty} \frac{\text{Im}\,\varphi(x+iy) - \lambda(\{\infty\})y}{y} = 0. 
\]
By combining these, we obtain the first convergence.

% for $x+iy \in V_\phi(c)$ satisfying $y \geq b$,

% $x_n + iy_n \in V_\phi$
% $x_n + iy_n \to \infty$
% \[
%   \displaystyle \frac{1+s^2}{(s-x_n)^2 + y_n^2} \leq \frac{1+s^2}{y_n^2} \leq
%   \frac{1+s^2}{(x_n^2+y_n^2) \cos^2\phi} \to 0 \ (n \to \infty)
% \]
% \[
%   \lim_{n \to \infty} \frac{\text{Im}\,\varphi(x_n+iy_n)}{y_n} = \lambda(\{\infty\})
% \]

To show the second convergence, introduce a new variable $w \in \C_+$ by
$w = -1/(z-c) \iff z = c - 1/w$. Then $z \in c+V_\phi \iff w \in V_\phi$ and
the condition $|z-c| \leq b$ is equivalent to $|w| \geq 1/b$.
Let % $\psi \in \End(\C_+)$ be defined by
$\psi(w) = \varphi(-1/(z-c))$ and $\mu$ be the representing measure of $\psi$.
Under the change-of-variable $s = c-1/t$ on $\overline{\R}$, the transformation formula shows 
\[
  \mu(dt) = \frac{t^2 + (ct-1)^2}{1+t^2} \lambda(ds) 
\]
and hence $\mu(\{\infty\}) = (1+c^2) \lambda(\{c\})$. Thus the target of the first convergence for $\psi$ takes the form
\[
  \frac{\psi(u+iv) - \mu(\{\infty\}) (u+iv)}{v}
  = \frac{1}{vy}  y \left( \varphi(x+iy) + \lambda(\{c\}) \frac{1+c^2}{x+iy-c} \right).
\]
Since $vy = y^2/((x-c)^2+y^2)$ satisfies $\cos^2\phi \leq vy \leq 1$, the first convergence for $\psi$ implies the second convergence
for $\varphi$ and we are done. 
\end{proof}

\begin{Corollary}\label{mass} We have the following expressions for atomic masses 
  \[
\lim_{y \to +\infty} \frac{\varphi(iy)}{y}  =  i\lambda(\{\infty\})
\]
and
  \[
    \lim_{y \to +0} y\varphi(c+iy) = i(1+c^2)\lambda(\{c\}). 
  \]
\end{Corollary}

\end{document}